\documentclass[12pt]{amsart}
\usepackage{latexsym,amsmath,amsthm,amssymb}
\usepackage[colorlinks,
            linkcolor=blue,
            anchorcolor=blue,
            citecolor=blue
            ]{hyperref}
\usepackage[hmargin=1.2in,vmargin=1.2in]{geometry}

\def \a {\alpha}
\def \b {\beta}
\def \d {\delta}
\def \e {\epsilon}

\def \s {\sigma}

\def \v {\nu}

\def \F {{\mathbb F}}

\def \deg {{\rm deg}}

\def \A {{\mathcal A}}
\def \G {{\mathcal G}}

\newtheorem{theorem}{Theorem}[section]
\newtheorem{corollary}[theorem]{Corollary}
\newtheorem{proposition}[theorem]{Proposition}

\newtheorem{example}{Example}
\newtheorem{remark}[theorem]{Remark}

\begin{document}

\title[Subfields of the Hermitian function field]{On  subfields of the Hermitian function fields involving the involution automorphism}
\thanks{The first author is partially supported by the National Natural Science Foundation of China under Grant 11501493 
and the State Scholarship Fund of China Scholarship Council.}
\author{Liming Ma}\address{School of Mathematical Sciences, Yangzhou University, Yangzhou China
225002}\email{lmma@yzu.edu.cn}
\author{Chaoping Xing} \address{Division of Mathematical Sciences, School of Physical \& Mathematical Sciences,
Nanyang Technological University, Singapore
637371}\email{xingcp@ntu.edu.sg}
\maketitle

\begin{abstract}
A  function field over a finite field is called maximal if it achieves the Hasse-Weil bound. Finding possible genera that maximal function fields achieve has both theoretical interest and practical applications to coding theory and other topics.  As a subfield of a maximal function field is also maximal, one way to find maximal function fields is to find all subfields of a maximal function field.  Due to the large automorphism group of the Hermitian function field, it is natural to find as many subfields of the Hermitian function field as possible. In literature, most of papers studied subfields fixed by subgroups of the decomposition group at one point (usually the point at infinity). This is because it becomes much more complicated to study the subfield fixed by a subgroup  that is not contained in the decomposition group at one point.
 In this paper, we study subfields of the Hermitian function field fixed by subgroups that are not contained in the decomposition group of any point except the cyclic subgroups. It turns out that some new  maximal function fields are found.
\end{abstract}

\section{Introduction}
Let $\mathbb F_\ell$ be a finite field with $\ell$ elements and $F/\mathbb F_\ell$ be an algebraic function field of one variable with the full constant field $\mathbb F_\ell$ with genus $g$.
If the number of rational places of $F$ attains the Hasse-Weil bound
$$N(F)\leq \ell+1+2g\sqrt{\ell},$$
then $F$ is said to be \emph{maximal}.
It follows that $F$ could be maximal only if either $g$ is zero or $\ell$ is a square.

The most important example of maximal function field is the
Hermitian function field $H /\mathbb F_{\ell}$ with $\ell=q^2$, where
$q$ is a prime power. The Hermitian function field $H$ over $\F_{q^2}$ is defined by the equation
$$y^{q}+y=x^{q+1}$$
with $H=\F_{q^2}(x,y)$.

The set of rational places of $H$ consists of the infinite place $P_{\infty}$ which is the unique common pole of $x$ and $y$ and $P_{\alpha,\beta}$ which is the unique common zero of $x-\alpha$ and $y-\beta$ for each $(\a,\beta) \in \mathbb{F}_{q^2}^2$ satisfying $\b^q+\b=\a^{q+1}$. Thus, it has $q^3+1$ rational places in total. The genus of the Hermitian function field is $(q^2-q)/2.$
Furthermore, the Hermitian function field is the unique maximal function field of genus $(q^2-q)/2$ over the finite field $\mathbb{F}_{q^2}$ (see \cite{RS94}).
The automorphism group $\mathcal{A}$ of the Hermitian function field is defined by
$$\mathcal{A}=\text{Aut}(H/\mathbb{F}_{q^2})=\{\sigma:H\mapsto H| \sigma \text{ is a } \mathbb{F}_{q^2}\text{-automorphism of } H\}.$$
This  automorphism group is extremely large and
isomorphic to the projective unitary group $\text{PGU}(3,q^2)$ with order $q^3(q^2-1)(q^3+1)$ (see \cite{St73}).
The decomposition group $\A(P_\infty)$ of the infinite place is equal to
\[\{\sigma \in \mathcal A :\;
\sigma(P_\infty)=P_\infty\}=\{\sigma=[a,b,c]:\; (a,b,c)\in \mathbb{F}_{q^2}^*\times\mathbb{F}_{q^2}\times\mathbb{F}_{q^2}, c^{q}+c=b^{q+1}\},
\] where
$\sigma=[a,b,c]$ is the automorphism defined by
$$
\begin{cases}
\sigma(x) =ax+b,\\
\sigma(y) = a^{q+1}y+ab^{q}x+c.
\end{cases}
 $$
Thus, the order of $\A(P_\infty)$ is $q^3(q^2-1)$. Apart from the automorphisms of $\A(P_\infty)$, there is  an {\it involution}  automorphism, denoted by $\omega$,  given by
$$\omega(x)=\frac{x}{y}, \qquad \omega(y)=\frac{1}{y}.$$
The order of $\omega$ is $2$.
Then the full automorphism
group of  $H$ is generated by $\mathcal A(P_\infty)$ and $\omega$, i.e.,
$\mathcal A=\langle \mathcal A(P_\infty), \omega \rangle.$

For a maximal function field $F/\mathbb F_{q^2}$,
any function field $E$ with $\mathbb F_{q^2} \subsetneq E \subseteq F$ is maximal as well (see \cite{La87}).
Hence, one can construct a large number of maximal function fields by considering the fixed subfields with respect to some subgroups of the automorphism group $\mathcal A$ of the Hermitian function fields (see \cite{AQ04, BMXY13, CK99, CKT99, FGT97, FT96, GSX00, GK09,HKT08, MXY16, NX01, XS95}).

The goal of this article to find out possible genera which are achieved by  maximal function fields. One natural way to realize this goal is to find as many subfields of $H$ as possible. This is equivalent to finding as many subgroups of ${\mathcal A}$ as possible. The subfields of the Hermitian function fields considered in the literature are usually those fixed by  subgroups of the decomposition group $\mathcal A(P_\infty)$ except for \cite{GSX00} where the subgroup is generated by $\omega$ and the subgroup $\{\sigma=[a,0,0]:\; a\in \mathbb{F}_{q^2}^*\}$ and the characteristic of $\mathbb{F}_{q^2}$ is odd, etc.  In particular,  it was showed in \cite{BMXY13} that we can obtain all the genera of maximal function fields from the fixed subfields of the subgroups of the $\mathcal A(P_\infty)$ for the odd characteristic case. In   this article, we consider various subgroups of  ${\mathcal A}$ that involve both the decomposition group $\mathcal A(P_\infty)$ and the involution automorphism $\omega$. It would be exciting to find all subgroups of ${\mathcal A}$ and genera of corresponding subfields. The current article  moves one step forward by considering the involution automorphism $\omega$. Most of the subgroups discussed in this paper are not contained in the decomposition group at any point except the cyclic subgroup. Thus, the subfields of the Hermitian function field obtained in this paper are new despite some genera of these subfields have  already been found in literature.

We summarize genera obtained in this paper in Tables I and II. In particular, we get some new genera  that are not achievable by  subgroups of the decomposition group $\mathcal A(P_\infty)$. The genera in Table 1 (i)-(iii) and Table 2 (ii) have not be found in literature, while the rest of the genera given in Tables 1 and 2 can be found in \cite{CKT99, GSX00}.    

{\footnotesize
\begin{table}
\caption{Genera for even characteristic}

\begin{tabular}{|c|c|c|c|}
\hline
No. &Conditions on parameter & Genera & References\\ \hline\hline
 (i)&$d=(m,q+1), \tilde{d}=(m,q-1)$ & $ [{q^2-q+m-(d-1)(q-1)-\tilde{d}(q+1)}]/({4m})$ & Thm \ref{a002m}\\
 \hline
(ii)& $m|(q-1)$ & $(q^2-q-mq)/(4m)$ & Thm  \ref{10c-} \\
\hline
(iii)& $m|(q+1)$ & $(q^2-q-mq+2m-2)/(4m)$  & Thm  \ref{10c+}\\
\hline
(iv)&  $m|(q^2-1), d=(m,q+1)$ & $(q-1)(q+1-d)/(2m)$  & Thm  \ref{a0c-}\\
\hline
 \end{tabular}

\end{table}

}
{\footnotesize
\begin{table}
\caption{Genera for odd characteristic}

\begin{tabular}{|c|c|c|c|}
\hline
No. &Conditions on parameter & Genera & References\\ \hline\hline
 (i)& $3\nmid (q+1), m|(q+1)$, $m$ is odd  &  $1+(q^2-q-2)/(2m)$ & Thm  \ref{o10c+} \\
\hline
 (ii)& $3\nmid (q+1), 4|m|(q+1), q\equiv 3(\text{mod } 8)$ & $1+(q^2-4q-5)/(2m)$  &Thm \ref{o10c+} \\
\hline
 (iii)& $3\nmid (q+1), m|(q+1)$ and (i), (ii) are not satisfied & $1+(q^2-2q-3)/(2m)$  &Thm  \ref{o10c+} \\
\hline
(iv)&  $m|2(q-1)$, $m$ is odd & $(q^2-q)/(2m)$  & Thm  \ref{o10c-}\\
\hline
(v)& $4|m|2(q-1), q\equiv 3(\text{mod } 4)$ & $(q^2-4q+3)/2m$  & Thm  \ref{o10c-} \\
\hline
 (vi) & $m|2(q-1)$ and No. (iv), (v) are not satisfied  & $(q^2-2q+1)/(2m)$  & Thm  \ref{o10c-} \\
  \hline
(vii)& $m|(q+1)$ & $(q-1)(q+1-m)/(2m)$  & Thm  \ref{oa0c-} \\
\hline
     (viii)&    $m|2(q+1), m\nmid (q+1)$ & $(q-1)(q+1-\frac{m}{2})/(2m)$  & Thm  \ref{oa0c-}  \\
\hline
\end{tabular}

\end{table}
}

\section{Preliminary}
Let ${\mathcal G}$ be a subgroup of the automorphism group $\mathcal{A}$ and let $H^{\mathcal G}$ be the fixed subfield of the 
Hermitian function field $H$ with respect to ${\mathcal G}$, i.e.,
$$H^{\mathcal G}=\{z\in H|\sigma(z)=z \text{ for all } \sigma \in {\mathcal G}\}.$$
Then $H/H^{\mathcal G}$ is a Galois extension of algebraic function fields with the Galois group ${\mathcal G}=\text{Gal}(H/H^{\mathcal G})$.
The Hurwitz genus formula yields
$$2g(H)-2=|{\mathcal G}|\cdot [2g(H^{\mathcal G})-2]+\deg \text{ Diff}(H/H^{\mathcal G}),$$
where $\text{Diff}(H/H^{\mathcal G})$ is the different of the extension $H/H^{\mathcal G}$.

Let $P$ be a place of $H$ and let $Q=P\cap H^{\mathcal G}$ be the restriction of $P$ to $H^{\mathcal G}$.
We denote by $d(P)=d(P|Q), \ e(P)=e(P|Q)$  the different exponent and ramification index of $P|Q$, respectively.
Then the different of $H/H^{\mathcal G}$ is given by
$$\text{Diff}(H/H^{\mathcal G})=\sum_{P\in H} d(P) P.$$
If $P|Q$ is unramified or tamely ramified, then $d(P)=e(P)-1$ by Dedekind's Different Theorem \cite[Theorem 3.5.1]{St09}.
The $i$-th ramification group ${\mathcal G}_i(P)$ of $P|Q$ for each $i\ge -1$ is defined by 
$${\mathcal G}_i(P)=\{ \sigma\in {\mathcal G}| v_P(\sigma(z)-z) \ge i+1 \text{ for all } z\in \mathcal{O}_P\},$$
where $v_P$ is the normalized discrete valuation of $H$ corresponding to the place $P$. 
If $P|Q$ is wildly ramified, that is, $e(P)$ is divisible by $\text{char}(\mathbb{F}_{q^2})$, then the different exponent $d(P)$ is
$$d(P)=\sum_{i=0}^{+\infty} \Big{(} |{\mathcal G}_i(P)|-1\Big{)}$$
by  Hilbert's Different Theorem \cite[Theorem 3.8.7]{St09}.

It has been showed that any ramified place $P\in \mathbb{P}_H$ in the extension $H/H^{\mathcal G}$ must be a rational place or a place of degree three (see \cite[Proposition 2.2]{GSX00}). Furthermore,  a place of degree three of $H$ in the extension $H/H^{\mathcal G}$  is unramified or tamely ramified, and has ramification index  being a divisor of $q^2-q+1$.
In particular, assume that $H/H^{\mathcal G}$ is tamely ramified and all places of degree $3$ of $H$ are unramified in $H/H^{\mathcal G}$.
If $P$ is tamely ramified, then $$d(P)=e(P)-1=\# \{\sigma \in {\mathcal G}\setminus \{1\} :\; \sigma(P)=P   \}.$$
Let $N(\sigma)$ be the number of rational places stabilized by the automorphism $\sigma$, that is,
$$N(\sigma)=\# \{P\in \mathbb{P}_H:\; \deg(P)=1 \text{ and } \sigma(P)=P\}.$$
Hence, the degree of the different of $H/H^{\mathcal G}$ is
$$\deg (\text{Diff}(H/H^{\mathcal G}))=\sum_{P\in \mathbb{P}_H, \deg P=1} (e(P)-1)=\sum_{1\neq \sigma \in {\mathcal G}} N(\sigma).$$

\section{The fixed subfields of $\langle [a,0,0], \omega \rangle$ in even characteristic}
In this section, we consider the group ${\mathcal C}$ which is generated by the automorphisms $\epsilon$ and $\omega$, where
$$\e(x)=ax,\quad \e(y)=a^{q+1}y\quad \text{ and }\quad \omega(x)=x/y,\quad \omega(y)=1/y.$$
Here $a$ is a primitive $(q^2-1)$-th root of unity. This subgroup was first discussed in \cite{GSX00} where only odd characteristic was considered. The reason why the even characteristic was not considered is that the extension is wildly ramified in  the case of even characteristic. Thus, we consider only the even characteristic case in this section.

Any $\sigma\in {\mathcal C}$ must be in the form $\e^i$ or $\omega \e^i$ for some $0\le i \le q^2-2$, since $ \e \omega=\omega \e^{-q}$.
Hence, $\text{ord}({\mathcal C})=2(q^2-1)$. The automorphisms in $\mathcal C$ can be given explicitly in the following form
$$\e^i(x)=a^ix, \quad \e^i(y)=a^{i(q+1)}y \quad  \text{ and } \quad  \omega \e^i(x)=\frac{a^ix}{y}, \quad \omega \e^i(y)=\frac{a^{i(q+1)}}{y}$$
for $0\le i \le q^2-2.$
In order to obtain the genus of the fixed subfield from the Hurwitz genus formula, we need to calculate the different exponent for each ramified place in the extension $H/H^{\mathcal C}$.

\begin{proposition}
\label{2diff}
Assume that $\text{char}(\mathbb{F}_{q^2})=2$.  Let $H$ be the Hermitian function field over $\mathbb{F}_{q^2}$ and 
let ${\mathcal C}$ be the group generated by the $\e$ and $\omega$. Then the different of the extension $H/H^{\mathcal C}$ is
$${\rm Diff}(H/H^{\mathcal C})=(q^2-2)P_{\infty}+(q^2-2)P_{0,0}+(3q+2)\sum_{\b\in \mathbb{F}_q^*}P_{0,\b}.$$
\end{proposition}

\begin{proof} 
First let us calculate the different exponent of each rational place of $H$.
For the infinity place $P_{\infty}$,
$$\sigma(P_{\infty})=P_{\infty} \Leftrightarrow \sigma \in \mathcal{A}(P_{\infty})\cap {\mathcal C}= \langle \e \rangle.$$
Then the order of the decomposition group of $P_{\infty}$ in $H/H^{\mathcal C}$ is $|{\mathcal G}_0(P_{\infty})|=q^2-1$. 
Hence,  the different exponent of $P_{\infty}$ in $H/H^{\mathcal C}$ is $d(P_{\infty})=e(P_{\infty})-1=q^2-2.$ 

For the place $P_{0,0}$, it is easy to see that $\e^i(P_{0,0})=P_{0,0}$ and $ \omega \e^i(P_{0,0})=P_{\infty}$ for any $0\le i\le q^2-2$.
Hence, $|{\mathcal G}_0(P_{0,0})|=q^2-1$ and the different exponent of $P_{0,0}$  in $H/H^{\mathcal C}$ is also $d(P_{0,0})=q^2-2.$

For a place $P_{0,\beta}$ with $\b^q+\b=0 \text{ and } \beta\neq 0$,
$$\e^i(P_{0,\beta})=P_{0,\beta}  \Leftrightarrow a^{i(q+1)}\b=\b  \Leftrightarrow (q-1)|i.$$ 
Moreover, $x$ is a prime element of $P_{0,\beta}$ and for $i=(q-1)k$ with $0< k\le q$,
$$v_{P_{0,\b}}(\e^i(x)-x) =v_{P_{0,\b}}(a^ix-x) =v_{P_{0,\b}}((a^i-1)x)=1.$$
On the other hand,
$$ \omega \e^i(P_{0,\beta})=P_{0,\beta}   \Leftrightarrow  a^{i(q+1)}/\b=\b   \Leftrightarrow  a^{i(q+1)}=\b^2.$$
Then there are exactly $q+1$ automorphisms $ \omega \e^i$ such that $ \omega \e^i(P_{0,\beta})=P_{0,\beta}$. 
If $\b$ runs through all elements of $\mathbb{F}_q^*$, then the solutions $i$ runs through the set of positive integers between $0$ and $q^2-2$. 
For such an integer $i$ satisfying with  $ \omega \e^i(P_{0,\beta})=P_{0,\beta}$, we have
$$v_{P_{0,\b}}\Big{(} \omega \e^i(x)-x\Big{)} =v_{P_{0,\b}}\Big{(}\frac{a^ix}{y}-x\Big{)} =v_{P_{0,\b}}\Big{(}\frac{y-a^i}{y}x\Big{)} =\begin{cases}
    q+2  & \text{if } \b=a^i,\\
      1& \text{otherwise.}
\end{cases}$$
The last equation holds true since the place $P_{0,\b}$ is the unique common zero of $x$ and $y-\b$, i.e., $v_{P_{0,\b}}(x)=1$ and  $v_{P_{0,\b}}(y-\b)=q+1$ which is obtained from  the defining equation $y^q+y=x^{q+1}$ and the Strict Triangle Inequality.
Hence, $|{\mathcal G}_0(P_{0,\b})|=2q+2, \ |{\mathcal G}_1(P_{0,\b})|= |{\mathcal G}_2(P_{0,\b})|=\cdots = |{\mathcal G}_{q+1}(P_{0,\b})|=2$ and $ |{\mathcal G}_{q+2}(P_{0,\b})|=1$.
By Hilbert's Different Theorem, the different exponent of $P_{0,\b}$ is 
$$d(P_{0,\b})=\sum_{i=0}^{+\infty}\Big{(}|{\mathcal G}_i(P_{0,\b})|-1\Big{)}=3q+2.$$

For a place $P_{\a,\b}$ with $\b^q+\b=\a^{q+1} \text{ and } \a\neq 0$,
$$\e^i(P_{\a,\beta})=P_{\a,\beta}   \Leftrightarrow  a^i \a=\a,\ a^{i(q+1)}\b=\b   \Leftrightarrow  i=0.$$ On the other hand,
$$ \omega \e^i(P_{\a,\beta})=P_{\a,\beta}   \Leftrightarrow  \frac{a^i \a}{\b}=\a, \ \frac{a^{i(q+1)}}{\b}=\b   \Leftrightarrow  \b=a^i,\ \b^2=a^{i(q+1)}.$$
Then $\b\in \mathbb{F}_q$ which is a contradiction to $\b^q+\b=\a^{q+1}\neq 0$.
Hence, $|{\mathcal G}_0(P_{\a,\b})|=1$ and the different exponent of $P_{\a,\b}$ is $d(P_{\a,\b})=0.$ 

By the Hurwitz genus formula, we have
$$q^2-q-2\ge 2(q^2-1)[2g(H^{\mathcal C})-2]+(q^2-2)+(q^2-2)+(3q+2)(q-1).$$
It follows that $g(H^{\mathcal C})\leq 0$. The genus of any function field must be a non-negative integer. Hence,  $g(H^{\mathcal C})=0$ and all places of degree three of $H$ are unramified in $H/H^{\mathcal C}$. This proposition follows immediately.
\end{proof}

\begin{theorem}
\label{a002m}
Assume that $\text{char}(\mathbb{F}_{q^2})=2$. Let $m$ be a divisor of $q^2-1$ and let $b\in \mathbb{F}_{q^2}^*$ be an element of order $m$. Consider the group ${\mathcal G}=\langle \lambda, \omega \rangle$  which is generated by the automorphism $\lambda$ and $\omega$, where $$\lambda(x)=bx, \quad \lambda(y)=b^{q+1}y   \quad \text{and}  \quad \omega(x)=\frac{x}{y}, \quad \omega(y)=\frac{1}{y}. $$ Let $d=\gcd(m,q+1), \tilde{d}=\gcd(m,q-1)$.
Then the genus of the fixed field $H^{\mathcal G}$ is
$$g(H^{\mathcal G})=\frac{q^2-q+m-(d-1)(q-1)-\tilde{d}(q+1)}{4m}.$$
\end{theorem}

\begin{proof}
The order of the group ${\mathcal G}$ is $2m$, and it consists of all automorphisms with the form
$$\sigma_c(x)=cx, \quad \sigma_c(y)=c^{q+1}y, \quad c^m=1$$ and
$$\tau_c(x)=\frac{cx}{y},  \quad \tau_c(y)=\frac{c^{q+1}}{y}, \quad  c^m=1.$$
We only need to calculate the different exponents for $P_{\infty}$ and $P_{0,\b}$ with $\b\in \mathbb{F}_q$ in the extension $H/H^{\mathcal G}$ from Proposition \ref{2diff}.

For the infinity place $P_{\infty}$, we have $\sigma(P_{\infty})=P_{\infty}    \Leftrightarrow  \sigma\in \mathcal{A}(P_{\infty}) \cap {\mathcal G}=\langle \lambda \rangle$. This implies that $|{\mathcal G}_0(P_{\infty})|=m$. Hence, the different exponent of $P_{\infty}$ is $d(P_{\infty})=m-1.$

For the place $P_{0,0}$, we also have $\sigma(P_{0,0})=P_{0,0}   \Leftrightarrow  \sigma\in \mathcal{A}(P_{\infty}) \cap {\mathcal G}=\langle \lambda \rangle$. 
Thus, $|{\mathcal G}_0(P_{0,0})|=m$ and $d(P_{0,0})=m-1.$

For a place $P_{0,\beta}$ with $\b^q+\b=0$ and $\b\neq 0$, we have
$$\sigma_c(P_{0,\beta})=P_{0,\beta}   \Leftrightarrow  c^{q+1}=1 \text{ and } c^m=1.$$ Then there are exactly $d=\gcd(m, q+1)$ automorphisms $\sigma_c$ such that $\sigma_c(P_{0,\beta})=P_{0,\beta}$. Furthermore,  $v_{P_{0,\b}}(\sigma_c(x)-x)=1$ for $c\neq 1$.
On the other hand,
$$\tau_c(P_{0,\beta})=P_{0,\beta}   \Leftrightarrow c^{q+1}=\beta^2, \quad \b^q+\b=0 \quad \text{and} \quad c^m=1.$$
For each fixed element $\b\in \mathbb{F}_q^*$, let $N(\b)$ denote the number of elements $c\in \mathbb{F}_{q^2}^*$ satisfy $c^{q+1}=\beta^2 \text{ and } c^m=1$. Then we obtain $\sum_{\b\in \mathbb{F}_q^*} N(\b)=m.$ 
Furthermore, we have
$$v_{P_{0,\b}}\Big{(}\tau_c(x)-x\Big{)} =v_{P_{0,\b}}\Big{(}\frac{cx}{y}-x\Big{)} =v_{P_{0,\b}}\Big{(}\frac{y-c}{y}x\Big{)} =\begin{cases}
    q+2  & \text{if } c=\b,\\
      1& \text{otherwise.}
\end{cases}$$
If $c=\b$, then $\b^{q-1}=1$ and $\b^{m}=1$. Then there are exactly $ \tilde{d}=\gcd(m,q-1)$ elements $\b\in \mathbb{F}_q^*$ such that the place $P_{0,\b}$ is wildly ramified in $H/H^{\mathcal G}$. For such a wildly ramified place $P_{0,\b}$, the orders of the higher ramification groups are $  |\mathcal{G}_i(P_{0,\beta}) |=2 \mbox{ for } 1\le i \le q+1 \mbox{ and }  |\mathcal{G}_{q+2}(P_{0,\beta}) |=1.$
Hence, $$\sum_{\b\in \mathbb{F}_q^*}|{\mathcal G}_0(P_{0,\beta})|=d(q-1)+m, \quad
\sum_{\b\in \mathbb{F}_q^*} (|{\mathcal G}_i(P_{0,\beta})|-1)=\tilde{d}$$
for each $1\le i\le q+1$ and $ |{\mathcal G}_{q+2}(P_{0,\beta})|=1$ for any $\b\in \mathbb{F}_q^*$.
Hence, the sum of different exponents for all the places $P_{0,\b}$ with $\b \in \mathbb{F}_q^*$ is
$$\sum_{\b\in \mathbb{F}_q^*}d(P_{0,\beta})=(d-1)(q-1)+m+\tilde{d}(q+1)$$
by Hilbert's Different Theorem.

For other places $P$, we have $d(P)=0.$ Hence, the degree of the different of $H/H^{\mathcal G}$ is $$\deg(\text{Diff}(H/H^{\mathcal G}))=m-1+m-1+(d-1)(q-1)+m+\tilde{d}(q+1).$$
The Hurwitz genus formula for $H/H^{\mathcal G}$ yields
$$q^2-q-2=2m\cdot [2g(H^{\mathcal G})-2]+\deg(\text{Diff}(H/H^{\mathcal G})).$$
Hence, this theorem follows immediately.
\end{proof}

\begin{corollary} Assume that $\text{char}(\mathbb{F}_{q^2})=2$  and $H$ is the Hermitian function field over the finite field $\mathbb{F}_{q^2}$.
\begin{itemize}
\item [(1)] For any divisor $m$ of $q-1$, there is a subfield $E\subseteq H$ of genus $$g(E)=\frac{q^2-q-mq}{4m}.$$
\item [(2)]  For any divisor $m$ of $q+1$, there is a subfield $E\subseteq H$ of genus $$g(E)=\frac{q^2-q-mq+2m-2}{4m}.$$
\end{itemize}
\end{corollary}
\begin{proof}
Let $\mathcal{G}$ be the group which is defined in Theorem \ref{a002m}. 
\begin{itemize}
\item [(1)] If $m|(q-1)$, then $d=\gcd(m,q+1)=1$ and $ \tilde{d}=\gcd(m,q-1)=m$. By Theorem \ref{a002m}, the genus of the fixed field $E=H^{\mathcal G}$ is $$g(E)=\frac{q^2-q-mq}{4m}.$$
\item [(2)] If $m|(q+1)$, then $d=\gcd(m,q+1)=m$ and $ \tilde{d}=\gcd(m,q-1)=1$. By Theorem \ref{a002m}, the genus of the fixed field $E=H^{\mathcal G}$ is $$g(E)=\frac{q^2-q-mq+2m-2}{4m}.$$
\end{itemize}
This completes the proof.
\end{proof}

\section{The fixed fields of subgroups of $\langle [1,0,c], \omega \rangle$}
In this section, we consider another group generated by the automorphisms $\omega$ and $\tau$, where $\tau$ is given by
$\tau(x)=x,\  \tau(y)=y+c$ for some $c\in \mathbb{F}_{q^2}$ satisfying $c^q+c=0.$
Let $\sigma=\tau \omega$, then $$\sigma(x)=\frac{x}{y+c}, \quad \sigma(y)=\frac{1}{y+c}.$$
The automorphism $\sigma^i$ can be given in the following form
$$\sigma^i(x)= \frac{x}{u_i y+v_i}, \quad \sigma^i(y)=  \frac{u_{i-1}y+v_{i-1}}{u_i y+v_i},$$
where $u_i, v_i$ satisfy the recursive relations $u_i=v_{i-1}$, $v_i=cv_{i-1}+u_{i-1}$ 
with the initial values $u_0=v_{-1}=0$ and $v_0=u_{-1}=1$.
Then it can be calculated that
$$v_n=\sum_{i=0}^{[\frac{n}{2}]}\binom{n-i}{i} c^{n-2i}.$$
However, it is difficult to find the order of $\sigma$ by using the above formula of $v_n$.
Hence, we rewrite the above recursive relations in the matrix representation $$\left(\begin{array}{c}v_i \\v_{i-1}\end{array}\right)=\left(\begin{array}{cc}c & 1 \\1 & 0\end{array}\right)  \left(\begin{array}{c}v_{i-1} \\v_{i-2}\end{array}\right).$$
Let $C$ be the matrix $\left(\begin{array}{cc}c & 1 \\1 & 0\end{array}\right).$
Then $$\left(\begin{array}{c}v_i \\v_{i-1}\end{array}\right)=C^i  \left(\begin{array}{c}v_{0} \\v_{-1}\end{array}\right)=C^i  \left(\begin{array}{c}1 \\0\end{array}\right).$$
The characteristic polynomial of the matrix $C$ is $$\det(xI_2-C)=\left|\begin{array}{cc}x-c & -1 \\-1 & x\end{array}\right|=x^2-cx-1.$$
Since the product of the two eigenvalues is $-1$, we may assume that the two eigenvalues are $x_1=\delta$ and $x_2=-\delta^{-1}$.
Then we have $$c=x_1+x_2=\delta-\delta^{-1}.$$ By the identity $c^q+c=0$, we have 
$$c^q+c=0 \Leftrightarrow \Big{(}\delta-\frac{1}{\delta}\Big{)}^q+\Big{(}\delta-\frac{1}{\delta}\Big{)}=0  \Leftrightarrow (\delta^{q+1}-1)(\delta^{q-1}+1)=0.$$
Hence, $\d^{q+1}=1$ or $\delta^{q-1}=-1.$

It is easy to calculate that the vector $\vec{p_1}=(\d,1)^T$ is an eigenvector of the eigenvalue $x_1=\d$ and the vector $\vec{p_2}=(-1,\d)^T$ is an eigenvector of the eigenvalue $x_2=-\d^{-1}$. If $x_1\neq x_2$, that is, $\d^2\neq -1$, then the matrix $P=\left(\begin{array}{cc}\delta & -1 \\1 & \delta\end{array}\right)$ is invertible and its inverse is given by
$$P^{-1}=\frac{1}{\delta^2+1}\left(\begin{array}{cc}\delta & 1 \\-1 & \delta\end{array}\right).$$
Hence, the matrix $C$ can be diagonalized to its eigenvalue matrix
$\Lambda=\left(\begin{array}{cc}\delta & 0 \\0 & - \delta^{-1} \end{array}\right)$ , that is,
$P^{-1}CP= \Lambda.$
Then
\begin{equation}
\label{ }\begin{split}
C^i &=(P\Lambda P^{-1})^i=P\Lambda^iP^{-1}\\
& = \frac{1}{\delta^2+1} \left(\begin{array}{cc}\delta & -1 \\1 & \delta\end{array}\right) \cdot \left(\begin{array}{cc}\delta & 0 \\0 & - \delta^{-1}\end{array}\right)^i \cdot \left(\begin{array}{cc}\delta & 1 \\-1 & \delta\end{array}\right)  \\
 &=\frac{1}{\delta^2+1} \left(\begin{array}{cc}\delta^{i+2}+(-\delta)^{-i}& \delta^{i+1}+(-\delta)^{-i+1} \\ \delta^{i+1}+(- \delta)^{-i+1}  & \delta^{i}+(- \delta)^{-i+2} \end{array}\right)
\end{split}
\end{equation}
Hence,  $v_i$ can also be given by the following formula $$v_i=\frac{\delta^{i+2}+(-\delta)^{-i}}{\delta^2+1}. $$
Then the order of the automorphism $\sigma$ can be determined as follows. 
Note that $$ \sigma^i=1 \Leftrightarrow \left(\begin{array}{c}v_i \\v_{i-1}\end{array}\right)=\left(\begin{array}{c}1 \\0\end{array}\right)  \Leftrightarrow  C^i \left(\begin{array}{c}1 \\0\end{array}\right)=\left(\begin{array}{c}1 \\0\end{array}\right) \Leftrightarrow  C^i=\left(\begin{array}{cc}1 & * \\0 & *\end{array}\right). $$
If $\delta^{i+1}+(-\delta)^{-i+1} =0$, that is,  $\d^{2i}=(-1)^i$, then $\delta^{i+2}+(-\delta)^{-i}=\delta^{i}+(- \delta)^{-i+2}$.
Hence,
$$\sigma^i=1 \Leftrightarrow C^i=I_2  \Leftrightarrow \left(\begin{array}{cc}\delta & 0 \\0 & -\delta^{-1}\end{array}\right)^i=I_2 \Leftrightarrow 
\d^i=1 \text{ and } (-1)^i=1.$$
If $\d^2=-1$, then the matrix $C$ cannot be diagonalized.

\subsection{Even characteristic}
In the even characteristic case, we have $c=\delta+\delta^{-1}$ and $$c^q+c=0 \Leftrightarrow (\delta^{q-1}-1)(\delta^{q+1}-1)=0  \Leftrightarrow \delta^{q-1}=1 \text{ or } \delta^{q+1}=1.$$
If $\d^2=-1$, then $\d=1$ and $c=0$. Hence, $\tau$ is the identity automorphism.
So we can assume $\delta$ as an element in $\mathbb{F}_{q^2}$ with an order $n$ which is  $q-1>1$ or $q+1$ in this subsection.

Now we need to determine the group structure of the group ${\mathcal D}$ generated by the two automorphisms $\tau$ and $\omega$.
Firstly, let us determine the order of the automorphism $\sigma=\tau \cdot \omega$. Note that
$$\sigma^i=1  \Leftrightarrow  v_{i-1}=0 \text{ and } v_{i}=1.$$ It has been showed that
$$v_i=\frac{1}{c}\Big{(}\delta^{i+1}+\frac{1}{\delta^{i+1}}\Big{)},$$
then $v_{i-1}=0  \Leftrightarrow \delta^i+\delta^{-i}=0 \Leftrightarrow \delta^{2i}=1 \Leftrightarrow \delta^{i}=1 \Leftrightarrow \text{ord}(\delta)|i$. Moreover, if $\text{ord}(\delta)|i$, then $v_i=1$. 
Hence, the order of the automorphism $\sigma$ is the same as the order of $\delta$.
Secondly, it can be directly verified that $\omega \sigma=\sigma^{n-1} \omega$. Hence,
$${\mathcal D}=\langle \tau,\omega \rangle=\langle \omega,\sigma| \omega^2=1, \sigma^n=1, \omega \sigma=\sigma^{n-1} \omega \rangle,$$
i.e., $\mathcal D$ is isomorphic to the Dihedral group $D_n$ of order $2n$.

Now we need to consider the ramification in the Galois extension $H/H^{\mathcal D}$.
Firstly, we consider the automorphism $\sigma^i$ for $1\le i \le n-1$,  then $v_{i-1}\neq 0$ and $\sigma^i(P_{\infty})\neq P_{\infty}$.
For any rational place $P_{\alpha,\beta}\in \mathbb{P}_H$,
$$\sigma(P_{\alpha,\beta})=P_{\alpha,\beta}  \Leftrightarrow
     \frac{\alpha}{v_{i-1}\beta+v_i}  =\alpha, \quad 
      \frac{v_{i-2}\beta+v_{i-1}}{v_{i-1}\beta+v_{i}}=\beta, \quad \beta^q+\beta=\alpha^{q+1}.$$
If $\alpha\neq 0$, then $v_{i-1}\beta+v_i=1$. Since $c^q+c=0$, we have $c\in \mathbb{F}_q$ and $v_i\in  \mathbb{F}_q$ which follows from the formula of $v_i$ in the variable $c$. Then
$\beta\in \mathbb{F}_q$
which contradicts to the third equation $\beta^q+\beta=\alpha^{q+1}\neq 0$.
Hence, $\alpha=0$ and $\b\in  \mathbb{F}_q$. 
From the second equation, $$v_{i-2}\b+v_{i-1}=v_{i-1} \beta^2 +v_i \beta  \Rightarrow v_{i-1}(\beta^2+c\beta+1) =0 \Rightarrow \beta^2+c\beta+1=0.$$
Then there are two roots $\beta=\delta$ and $\beta=\delta^{-1}$. \\
If $n=q-1$, then  $\d^q+\d=0$ and $(\d^{-1})^q+\d^{-1}=0$.
Hence, the two places $P_{0,\delta}$ and $P_{0,  \delta^{-1}}$ are stabilized by the automorphism $\sigma^i$.
Furthermore,
$$v_{P_{0,\delta}}\Big{(}\sigma^i(x)-x\Big{)} =v_{P_{0,\delta}}\Big{(}\frac{x}{v_{i-1}y+v_i}-x\Big{)}
=v_{P_{0,\delta}}\Big{(}\frac{v_{i-1}y+v_i+1}{v_{i-1}y+v_i} x\Big{)}=1,$$
since $v_{i-1}\delta+v_i+1=1+\delta^{-i}\neq 0$.\\
If $n=q+1$, then $\d^q+\d \neq 0$ and $(\d^{-1})^q+\d^{-1}\neq 0$ which contradict to $\a=0$ in the third equation. Hence, $P_{\a,\b}$ can't be stabilized by the automorphism $\s^i$ for any $1\le i \le q$.

Secondly, we consider the automorphism $\sigma^i \omega$ for $0\le i \le n-1$,
$$\sigma^i \omega(x) =\frac{x}{v_{i-2}y+v_{i-1}}, \quad \sigma^i \omega(y) =\frac{v_{i-1}y+v_{i}}{v_{i-2}y+v_{i-1}}.$$
The order of the automorphism $\sigma^i \omega$ is $2$, since $\omega \sigma=\sigma^{n-1} \omega$.
For $i=0$, we know $\omega(P_{0,0})=P_{\infty}$. If $\beta\neq 0$, then 
$$\omega(P_{\alpha,\beta})=P_{\alpha,\beta}   \Leftrightarrow  \frac{\alpha}{\beta}=\alpha, \quad \frac{1}{\beta}=\beta,  \quad \beta^q+\beta=\alpha^{q+1}.$$
Hence, $P_{0,1}$ is the unique rational place stabilized by the automorphism $\omega$.
Furthermore, $$v_{P_{0,1}}\Big{(}\omega(x)-x\Big{)}=v_{P_{0,1}}\Big{(}\frac{x}{y}-x\Big{)}=v_{P_{0,1}}\Big{(}\frac{y+1}{y}x\Big{)}=q+2.$$
For $i=1$, then $\sigma \omega=\tau$, i.e., $\tau(x)=x, \ \tau(y)=y+c$. Hence, $P_{\infty}$ is the unique rational place stabilized by the automorphism $\tau$ and $$v_{P_{\infty}}\Big{(}\tau\big{(}\frac{x}{y}\big{)}-\frac{x}{y}\Big{)}=v_{P_{\infty}}\Big{(}\frac{x}{y+c}-\frac{x}{y}\Big{)}=v_{P_{\infty}}\Big{(}\frac{cx}{y(y+c)}\Big{)}=q+2.$$
For $2\le i\le n-1$, then $v_{i-2}\neq 0$ and $\sigma^i \omega(P_{\infty})\neq P_{\infty}$. 
For the place $P_{\alpha,\beta}\in \mathbb{P}_H$,
$$\sigma^i \omega(P_{\alpha,\beta})=P_{\alpha,\beta} \Leftrightarrow
\frac{\alpha}{v_{i-2}\beta+v_{i-1}}=\alpha,  \quad  \frac{v_{i-1}\beta+v_{i}}{v_{i-2}\beta+v_{i-1}}=\beta,  \quad  \beta^q+\beta=\alpha^{q+1}.$$
If $\alpha\neq 0$, then $v_{i-2}\beta+v_{i-1}=1$. Since $c\in \mathbb{F}_q$, we know $v_{i-2}, v_{i-1}\in \mathbb{F}_q$. It follows that
$\beta\in \mathbb{F}_q$ which contradicts to the third equation $\beta^q+\beta=\alpha^{q+1}\neq 0$.
Hence, $\alpha=0$, $\beta \in  \mathbb{F}_q$ and $$v_{i-1}\beta+v_i=v_{i-2}\beta^2+v_{i-1}\beta \Leftrightarrow \beta^2=\frac{v_i}{v_{i-2}}.$$
Moreover, it is easy to verify that $v_i/v_{i-2}$ are pairwise distinct elements in $\mathbb{F}_q$ for $2\le i \le n-1$.
Hence, the place $P_{0,\beta_i}$ with $\beta_i=(\frac{v_i}{v_{i-2}})^{\frac{q}{2}}$ is the unique rational place stabilized by $\sigma^i \omega$. Furthermore,
 $$v_{P_{0,\beta_i}}\Big{(}\sigma^i \omega(x)-x\Big{)}  =v_{P_{0,\beta_i}}\Big{(}\frac{x}{v_{i-2}y+v_{i-1}}-x\Big{)} 
 =v_{P_{0,\beta_i}}\Big{(}\frac{v_{i-2}y+v_{i-1}+1}{v_{i-2}y+v_{i-1}}x\Big{)}=q+2.$$
The last equation holds true, since $v_{i-2}\beta_i+v_{i-1}+1=0  \Leftrightarrow v_{i-2}^2{\beta_i}^2+v_{i-1}^2+1=0  \Leftrightarrow v_{i-2}v_i+v_{i-1}^2=1$ which can be obtained from the formula of $v_i$ in the variable $\d$.

\begin{theorem}
\label{10c-}
Let $H$ be the Hermitian function field over $\mathbb{F}_{q^2}$ with even characteristic, 
let $m$ be a divisor of $q-1$ and let ${\mathcal G}$ be the subgroup $\langle \omega, \s^{\frac{q-1}{m}} \rangle $ of ${\mathcal D}$. Then the genus of the fixed field $H^{\mathcal G}$ is $$g(H^{\mathcal G})=\frac{q^2-q-mq}{4m}.$$
\end{theorem}
\begin{proof}
The ramification groups of $P_{\infty}$ in  $H/H^{\mathcal G}$ are given by $$\G_0(P_{\infty})=\G_1(P_{\infty})=\cdots =\G_{q+1}(P_{\infty})=\langle \tau \rangle \text{ and } \G_{q+2}(P_{\infty})=\{id\}.$$
Hence, the different exponent of $P_{\infty}$ in $H/H^{\mathcal G}$ is $$d(P_{\infty})=\sum_{i=0}^{+\infty} (|\G_i(P_{\infty})|-1)=q+2.$$
It is easy to verify that $v_i/v_{i-2}\neq \d^2 \text{ or } \d^{-2}$, that is to say, $P_{0,\b_i}\neq P_{0,\d} \text{ or } P_{0,\delta^{-1}}$ for $2\le i \le q-2$. Then the different exponents of $P_{0,1}$ and $P_{0,\b_i}$ for $2\le i\le q-2$ are also
$$d(P_{0,1})=d(P_{0,\b_i})=q+2.$$
The decomposition groups of $P_{0,\delta}$ and $P_{0,\delta^{-1}}$ are $\langle \sigma \rangle$ which is a group of order $q-1$. Hence, the different exponents of $P_{0,\delta}$ and $P_{0,\delta^{-1}}$ are $$d(P_{0,\delta}) =d(P_{0,\delta^{-1}})=q-2.$$
By the Hurwitz genus formula, we have
$$q^2-q-2\geq 2(q-1)[2g(H^{\mathcal D})-2]+2\cdot (q-2)+(q-1)\cdot (q+2).$$
Hence, $g(H^{\mathcal D})=0$ and all places of degree three of $H$ are unramified in $H/H^{\mathcal D}$. 
The order of the subgroup ${\mathcal G}$ is $2m$. Then this theorem follows from the Hurwitz genus formula
$$q^2-q-2=2m\cdot [2g(H^{\mathcal G})-2]+2\cdot (m-1)+m\cdot (q+2).$$
\end{proof}

\begin{theorem}
\label{10c+}
Let $H$ be the Hermitian function field over $\mathbb{F}_{q^2}$ with even characteristic,
let $m$ be a divisor of $q+1$ and let ${\mathcal G}$ be the subgroup $\langle \omega, \s^{\frac{q+1}{m}} \rangle$ of ${\mathcal D}$. Then the genus of the fixed field  $H^{\mathcal G}$ is   $$g(H^{\mathcal G})=\frac{q^2 - q - m q+2m-2}{4m}.$$
\end{theorem}
\begin{proof}
The different exponents of the rational places $P_{\infty} \text{ and } P_{0,\b}$ with $\b\in \mathbb{F}_q$ in the extension $H/H^{\mathcal D}$ are $$d(P_{\infty})=d(P_{0,\b})=q+2.$$
By the Hurwitz genus formula, we have
$$q^2-q-2\geq 2(q+1)[2g(H^{\mathcal D})-2]+(q+1)\cdot (q+2).$$
Hence, $g(H^{\mathcal D})=0$ and all places of degree three of $H$ are unramified in $H/H^{\mathcal D}$. 
The order of the subgroup ${\mathcal G}$ is $2m$. Then this theorem follows from the Hurwitz genus formula
$$q^2-q-2=2m\cdot [2g(H^{\mathcal G})-2]+m\cdot (q+2).$$
\end{proof}

In this sub-subsection,  
we choose $\d$ with the maximal order $q-1$ or $q+1$ satisfying $c^q+c=0$. However, the order of $\d$ may just be a positive divisor of $q-1$ or $q+1$. Here are some examples, such as the order of $\d$ is $3$ which divides $q-1$ in Example \ref{101eo} and the order of $\d$ is $5$ in Example \ref{10c5}. Since the calculations are similar to the Theorem \ref{10c-} and \ref{10c+}, we omit the details.

\begin{example}
\label{101eo}
Let ${\mathcal G}$ be the group generated by the automorphisms $\tau$ and $\omega$ which are given by
$\tau(x)=x,  \ \tau(y)=y+1 \text{ and } \omega(x)=x/y,  \ \omega(y)=1/y.$
If $q=2^{2k}$, then the genus of the fixed subfield $H^{\mathcal G}$ is $$g(H^{\mathcal G})=\frac{q^2-4q}{12}.$$
\end{example}

\begin{example}
\label{10c5}
Let ${\mathcal G}$ be the group generated by the automorphisms $\omega$ and $\tau$ which is given by $\tau(x)=x,  \ \tau(y)=y+c$  with $c^2+c+1=0$ and $c^q+c=0$.
If $q=4^{2k}$, then the genus of the fixed subfield $H^{\mathcal G}$ is $$g(H^{\mathcal G})=\frac{q^2-6q}{20}.$$
Otherwise, the genus of the fixed subfield $H^{\mathcal G}$ is $$g(H^{\mathcal G})=\frac{q^2-6q+8}{20}.$$
\end{example}

\subsection{Odd characteristic}
In the odd characteristic case, we obtain $c=\delta-\delta^{-1}$ and 
$$c^q+c=0 \Leftrightarrow (\delta^{q+1}-1)(\delta^{q-1}+1)=0 \Leftrightarrow \delta^{q+1}=1 \text{ or } \delta^{q-1}=-1.$$
Hence,  we can fix $\delta$ as an element in $\mathbb{F}_{q^2}$ with an order  $n$ which is $q+1$ or  $2(q-1)$ in this subsection.
If $\d^2=-1$, then $\text{ord}(\d)=4.$ Hence, we also assume that $q\neq 3$ in this subsection for simplicity.

The automorphism $\sigma^i$ is given by
$$\sigma^i(x)=\frac{x}{v_{i-1}y+v_i},  \quad \sigma^i(y)=\frac{v_{i-2}y+v_{i-1}}{v_{i-1}y+v_{i}}.$$
We have shown that $\s^i=1 \Leftrightarrow \d^i=1 \text{ and } (-1)^i=1$. Hence, the order of the automorphism $\sigma$ is the same as the order of $\d$.
Let ${\mathcal D}$ be the group generated by the automorphism $\sigma$. Then the Galois extension $H/H^{\mathcal  D}$ is tamely ramified.

\subsubsection{$\text{ord}(\d)=q+1$} The order the automorphism $\sigma$ is $q+1$. We assume that $3\nmid (q+1)$ in this sub-subsection. 
As $\gcd(q+1, q^2-q+1)=\gcd(q+1, 3)=1$, then all places of degree $3$ in $H$ are unramified in $H/H^{\mathcal D}$.

For the infinite place $P_{\infty}$ and $1\le i \le q$, 
$$\s^i(P_{\infty})=P_{\infty}  \Leftrightarrow v_{i-1}=0   \Leftrightarrow  \delta^{i+1}+(-1)^{i-1}\delta^{-i+1}=0 \Leftrightarrow \delta^{2i}=(-1)^i.$$
If $i$ is even, then 
$\delta^{2i}=1 \Leftrightarrow 2i=q+1 \Leftrightarrow i=(q+1)/2 \text{ is even } \Leftrightarrow q\equiv 3(\text{mod} 4).$
Otherwise,
$\delta^{2i}=-1  \Leftrightarrow 2i=(q+1)/2 \text{ or } 3(q+1)/2  \Leftrightarrow i=(q+1)/4 \text{ or } 3(q+1)/4 \text{ is odd } 
 \Leftrightarrow q\equiv 3(\text{mod} 8).$
 
For any rational place $P_{\a,\b}\in \mathbb{P}_H$ and $1\le i\le q$,
$$\sigma^i(P_{\a,\b})=P_{\a,\b} \Leftrightarrow \frac{\a}{v_{i-1}\b+v_i}=\a,  \quad \frac{v_{i-2}\b+v_{i-1}}{v_{i-1}\b+v_{i}}=\b,  \quad \b^q+\b=\a^{q+1}.$$
{\bf Case 1:} $\a\neq 0$.\\
From the first equation, we have $v_{i-1}\b+v_i=1$. If $v_{i-1}=0$, then $v_i=1$. It follows that $\text{ord}(\sigma)|i$ which is impossible for $1\le i\le q$. 
Hence, $v_{i-1}\neq 0$ and this implies that $v_{i-2}\b+v_{i-1}=v_{i-1}\b^2+v_i\b  \Leftrightarrow v_{i-1}(\b^2+c\b-1)=0   \Leftrightarrow \b^2+c\b-1=0.$
There are two solutions $\b=-\delta \text{ or } \b=\delta^{-1}.$\\
Substitute $\beta=-\delta$ into the first equation, then we have
$-v_{i-1}\delta+v_i=(-1)^i \delta^{-i}.$
If $q\equiv 1(\text{mod} 4)$, then $-v_{i-1}\delta+v_i=1  \Leftrightarrow   \delta^i=(-1)^i   \Leftrightarrow  i=(q+1)/2.$
Furthermore, $\a^{q+1}=(-\d)^q+(-\d)=-(\d^q+\d)\neq 0$.
Hence, the $q+1$ places $P_{\a,-\d}$ with $\a^{q+1}=-(\d^q+\d)$ are stabilized by the automorphism $\sigma^{\frac{q+1}{2}}$. 
If $q\equiv 3(\text{mod} 4)$, then $-v_{i-1}\delta+v_i=(-1)^i \delta^{-i}\neq 1$.
Substitute $\beta=\delta^{-1}$ into the first equation similarly, then we have $v_{i-1}\delta^{-1}+v_i=\delta^{i}\neq 1.$ Hence, the place $P_{\a,\delta^{-1}}$ can't be stabilized by the automorphism $\sigma^i$.\\
{\bf Case 2:} $\a= 0$.\\
In this case, we have $\b^q+\b=0$ and $v_{i-2}\b+v_{i-1}=v_{i-1}\b^2+v_i\b  \Leftrightarrow v_{i-1}(\b^2+c\b-1)=0.$
If $v_{i-1}=0,$ then $\sigma^i(P_{0,\b})=P_{0,\b}$ for any $\b$ with $\b^q+\b=0$. \\
If $v_{i-1}\neq 0$, then $\b^2+c\b-1=0$. Hence, $\b=-\delta$ or $\b=\delta^{-1}$.
Furthermore,
$(-\delta)^q+(-\delta)=-\delta^q-\delta\neq 0$  and $(\delta^{-1})^q+\delta^{-1}\neq 0$
if $\delta^4\neq 1$, that is, $q\neq 3$. Hence, the place $P_{0,\b}$ can't be stabilized by the automorphism $\sigma^i$ with $v_{i-1}\neq 0$.

If $q\equiv 1(\text{mod } 4)$, then $v_{i-1}\neq 0$ for every $1\le i\le q$. Hence, $$\text{Diff}(H/H^{\mathcal D})=\sum_{\a^{q+1}=-\d^q-\d} P_{\a,\d}.$$
If $q\equiv 7(\text{mod } 8)$, then $v_{\frac{q+1}{2}-1}= 0$. Then the places $P_{\infty}$  and $P_{0,\b}$ with $\b^q+\b=0$ are stabilized by the automorphism $\sigma^{\frac{q+1}{2}}$. Hence, $$\text{Diff}(H/H^{\mathcal D})=P_{\infty}+\sum_{\b^q+\b=0} P_{0,\b}.$$
If $q\equiv 3(\text{mod } 8)$, then $v_{\frac{q+1}{4}-1}, v_{\frac{q+1}{2}-1}$ and $v_{\frac{3(q+1)}{4}-1}$ are equal to $0$. Then the  places  $P_{\infty}$ and $P_{0,\b}$ with $\b^q+\b=0$ are stabilized by the automorphisms $\sigma^{\frac{q+1}{4}}, \sigma^{\frac{q+1}{2}}$ and $\sigma^{\frac{3(q+1)}{4}}$. Hence, $$\text{Diff}(H/H^{\mathcal D})=3P_{\infty}+3\sum_{\b^q+\b=0} P_{0,\b}.$$

\begin{theorem}
\label{o10c+}
Let $H$ be the Hermitian function field over $\mathbb{F}_{q^2}$ with odd characteristic. Assume that $3\nmid (q+1)$. 
Let $m$ be a positive divisor of $q+1$ and  let ${\mathcal G}$ be the group generated by the automorphism $\sigma^{\frac{q+1}{m}}$. Then the genus of the fixed field $H^{\mathcal G}$ is
$$g(H^{\mathcal G})=
\begin{cases}
  1+ (q^2-q-2)/2m & \text{if } m \text{ is odd}, \\
   1+ (q^2-4q-5)/2m & \text{if } 4|m \text{ and } q\equiv 3(\text{mod } 8), \\
   1+ (q^2-2q-3)/2m & \text{otherwise}.
\end{cases}$$
\end{theorem}
\begin{proof}
It is easy to check that
$\sigma^{\frac{q+1}{2}} \in {\mathcal G} \Leftrightarrow m$ is even.
If $q\equiv 1 (\text{mod } 4)$ or  $q\equiv 7 (\text{mod } 8)$, then $$q^2-q-2=m[2g(H^{\mathcal G})-2]+\begin{cases}
    q+1  &    \text{if m is even}, \\
     0 &  \text{if m is odd}.
\end{cases}$$
If $q\equiv 3 (\text{mod } 8)$, then the automorphisms $\sigma^{\frac{q+1}{4}}, \sigma^{\frac{3(q+1)}{4}}\in {\mathcal G} \Leftrightarrow 4|m.$
This theorem follows from the Hurwitz genus formula,
$$q^2-q-2=m[2g(H^{\mathcal G})-2]+\begin{cases}
 0 &  \text{if m is odd},\\
   3(q+1)  &   \text{if } 4|m, \\
   q+1  &    \text{otherwise}.
\end{cases}$$
\end{proof}

\begin{remark}
We assume that $3\nmid (q+1)$ in Theorem \ref{o10c+}, since all places of degree $3$ of $H$ are unramified in  the extension  $H/H^{\mathcal G}$ under this assumption. In fact, we can't determine whether the places of degree $3$ of $H$ are ramified  in $H/H^{\mathcal G}$ or not  in the case of $ 3 | (q+1)$. 
The similar case occurs in Theorem \ref{oa0c+} as well.
\end{remark}

\subsubsection{$\text{ord}(\delta)=2(q-1)$}
If $\text{ord}(\delta)=2(q-1)$, then the order of the automorphism $\sigma$ is $2(q-1).$ 
As $\gcd(2q-2, q^2-q+1)=1$, all places of degree $3$ of $H$ are unramified in $H/H^{\mathcal D}$.

For the infinite place $P_{\infty}$ and $1\le i\le 2q-3$,,
$$\s^i(P_{\infty})=P_{\infty}  \Leftrightarrow v_{i-1}=0   \Leftrightarrow  \delta^{i+1}+(-1)^{i-1}\delta^{-i+1}=0 \Leftrightarrow \delta^{2i}=(-1)^i.$$
If $i$ is even, then $\delta^{2i}=1 \Leftrightarrow i=q-1$. 
If $i$ is odd, then
$\delta^{2i}=-1  \Leftrightarrow 2i=q-1 \text{ or } 3(q-1)
 \Leftrightarrow i=(q-1)/2 \text{ or } 3(q-1)/2 \text{ is odd } 
 \Leftrightarrow q\equiv 3(\text{mod} 4).$

For the place $P_{\a,\b}$ and $1\le i\le 2q-3$,
$$\sigma^i(P_{\a,\b})=P_{\a,\b} \Leftrightarrow \frac{\a}{v_{i-1}\b+v_i}=\a, \quad \frac{v_{i-2}\b+v_{i-1}}{v_{i-1}\b+v_{i}}=\b, \quad \b^q+\b=\a^{q+1}.$$
If $\alpha\neq 0$, then $v_{i-1}\b+v_i=1$. Suppose that $v_{i-1}=0$, then $v_i=1$. It follows that $\text{ord}(\sigma)|i$ which is impossible for $1\le i\le 2q-3$. Hence, 
 $v_{i-1}\neq 0$ and $\b^2+c\b-1=0$. There are two solutions $\b=-\delta$ or $\b=\delta^{-1}$.
It is easy to check that
$v_{i-1}\b+v_i=-v_{i-1}\d+v_i=(-1)^i \delta^{-i}\neq 1,$
since $\delta^{q-1}=-1\neq (-1)^{q-1}.$ Moreover, $v_{i-1}\delta^{-1}+v_i=\delta^i\neq 1.$
Hence, $\sigma^i(P_{\a,\b})\neq P_{\a,\b} \text{ for } \a\neq 0.$\\
Otherwise, $\b^q+\b=0$ and $$v_{i-2}\b+v_{i-1}=v_{i-1}\b^2+v_i\b  \Leftrightarrow v_{i-1}(\b^2+c\b-1)=0.$$
If $v_{i-1}=0$, then the places $P_{0,\b}$ with $\b^q+\b=0$ are stabilized by $\sigma^i$. 
Otherwise, $\b^2+c\b-1=0$ has two solutions $\b=-\delta$ of $\b=\delta^{-1}.$
It is easy to check that $(-\delta)^q+(-\delta)=-(\delta^q+\delta)=0 \text{ and } (\delta^{-1})^q+(\delta^{-1})=0,$ since $\delta^{q-1}=-1.$
Hence, the two places $P_{0,-\delta}$ and $P_{0,\delta^{-1}}$ are stabilized by the automorphism $\sigma^i$ with $v_{i-1}\neq 0$.

If $q\equiv 1(\text{mod } 4)$, then the places $P_{\infty}$ and $P_{0,\b}$ with $\b^q+\b=0$ are stabilized by the automorphism $\sigma^{q-1}$. Hence, the different of $H/H^{\mathcal D}$ is $$\text{Diff}(H/H^{\mathcal D})=P_{\infty}+\sum_{\b^q+\b=0} P_{0,\b}+(2q-4)(P_{0,-\d}+P_{0,\delta^{-1}}).$$
If $q\equiv 3(\text{mod } 4)$,  then the places  $P_{\infty}$ and $P_{0,\b}$  with $\b^q+\b=0$ are stabilized by the automorphism $\sigma^{\frac{q-1}{2}}$, $\sigma^{q-1}$ and $\sigma^{\frac{3(q-1)}{2}}.$
Hence, the different of $H/H^{\mathcal D}$ is $$\text{Diff}(H/H^{\mathcal D})=3P_{\infty}+3\sum_{\b^q+\b=0} P_{0,\b}+(2q-6)(P_{0,-\d}+P_{0,\delta^{-1}}).$$

\begin{theorem}
\label{o10c-}
Let $H$ be the Hermitian function field over $\mathbb{F}_{q^2}$ with odd characteristic, let $m$ be a positive divisor of $2(q-1)$ and  let 
${\mathcal G}$ be the group generated by the automorphism $\sigma^{\frac{2(q-1)}{m}}$. Then the genus of the fixed field 
$H^{\mathcal G}$ is
$$g(H^{\mathcal G})=\begin{cases}
   (q^2-q)/2m & \text{if } m \text{ is odd}, \\
    (q^2-4q+3)/2m & \text{if } 4|m \text{ and } q\equiv 3(\text{mod } 4), \\
   (q^2-2q+1)/2m &   \text{otherwise}.
\end{cases}$$
\end{theorem}

\begin{proof}
It is easy to check that
$\sigma^{q-1} \in {\mathcal G} \Leftrightarrow m$ is even.
If $q\equiv 1 (\text{mod } 4)$, then $$q^2-q-2=m[2g(H^{\mathcal G})-2]+\begin{cases}
    q+1+2(m-2)  &    \text{if } m \text{ is even}, \\
     2(m-1) &  \text{if } m \text{ is odd}.
\end{cases}$$
If $q\equiv 3 (\text{mod } 4)$, then the automorphisms $\sigma^{\frac{q-1}{2}}, \sigma^{\frac{3(q-1)}{2}}\in {\mathcal G} \Leftrightarrow 4|m.$
This theorem follows from the Hurwitz genus formula,
$$q^2-q-2=m[2g(H^{\mathcal G})-2]+\begin{cases}
 2(m-1) &  \text{if } m \mbox{ is odd},\\
  3(q+1)+2(m-4)  &   \text{if } 4|m, \\
   q+1+2(m-2)  &    \text{otherwise}.
\end{cases}$$
\end{proof}

\section{The fixed subfields of subgroups of $\langle [a,0,c], \omega \rangle$}
Let $\tau$ be an automorphism of the Hermitian function field $H$ over $\mathbb{F}_{q^2}$ with the form
$$\tau(x)=a x, \quad \tau(y)=a^{q+1}y+c$$ where $c^q+c=0$ and $a$ is a  $(q^2-1)$-th primitive element of the finite field $\mathbb{F}_{q^2}$.
Let $\sigma=\tau \omega$, then $$\sigma(x)=\frac{ax}{a^{q+1}y+c},  \quad \sigma(y)=\frac{1}{a^{q+1}y+c}.$$
The automorphism $\sigma^i$ can be given in the following form
$$\sigma^i(x)= \frac{a^ix}{u_iy+v_i},  \quad \sigma^i(y)=  \frac{u_{i-1}y+v_{i-1}}{u_iy+v_i}$$
where $u_i, v_i$ satisfy the recursive relations $u_i=a^{q+1}v_{i-1}$, $v_i=cv_{i-1}+u_{i-1}$ with the initial values $u_0=v_{-1}=0,\ v_0=1$.
The above recursive relations can be rewritten in the matrix representation $$\left(\begin{array}{c}v_i \\v_{i-1}\end{array}\right)=\left(\begin{array}{cc}c & a^{q+1} \\1 & 0\end{array}\right)  \left(\begin{array}{c}v_{i-1} \\v_{i-2}\end{array}\right).$$
Let $C=\left(\begin{array}{cc}c & a^{q+1} \\1 & 0\end{array}\right).$
Then the characteristic polynomial of the matrix $C$ is
$$\det(xI_2-C)=\left|\begin{array}{cc}x-c & -a^{q+1} \\-1 & x\end{array}\right|=x^2-cx-a^{q+1}.$$
Assume that the two eigenvalues are $x_1=\delta$ and $x_2=-\frac{a^{q+1}}{\delta}$. Then we have
$$c=\delta-\frac{a^{q+1}}{\delta}.$$
By the identity $c^q+c=0$, we have 
$$c^q+c=0 \Leftrightarrow  (\delta-\frac{a^{q+1}}{\delta})^q+(\delta-\frac{a^{q+1}}{\delta})=0
 \Leftrightarrow (\delta^{q+1}-a^{q+1})(\delta^{q-1}+1)=0.$$
Hence, $ \delta^{q+1}=a^{q+1} \text{ or } \delta^{q-1}=-1$.
 
If $\d^2\neq -a^{q+1}$, then there exists an invertible matrix $P=\left(\begin{array}{cc}\delta & -a^{q+1} \\1 & \delta\end{array}\right)$ such that $C$ is similar to the diagonal matrix $\Lambda=\left(\begin{array}{cc}\delta & 0 \\0 & -\frac{a^{q+1}}{\delta}\end{array}\right)$ , that is,
$P^{-1}CP= \Lambda.$
Then we can calculate that
\begin{equation}
\label{ }\begin{split}
C^i &=(P\Lambda P^{-1})^{i}=P\Lambda^iP^{-1}\\
 &=\frac{1}{\delta^2+a^{q+1}}\left(\begin{array}{cc}\delta & -a^{q+1} \\1 & \delta\end{array}\right) \cdot \left(\begin{array}{cc}\delta & 0 \\0 & -\frac{a^{q+1}}{\delta}\end{array}\right)^i \cdot \left(\begin{array}{cc}\delta & a^{q+1}\\ -1 & \delta\end{array}\right)  \\
 &=\frac{1}{\delta^2+a^{q+1}} \left(\begin{array}{cc}\delta^{i+2}+(-\d)^{-i} a^{(i+1)(q+1)} & a^{q+1}\delta^{i+1}+(-\d)^{-i+1} a^{(i+1)(q+1)} \\ \delta^{i+1}+(-\d)^{-i+1}  a^{i(q+1)}   & a^{q+1}\delta^{i}+(-\d)^{-i+2}  a^{i(q+1)} \end{array}\right).
\end{split}
\end{equation}
Hence, $v_i$ can be given by the following formula $$v_i=\frac{\delta^{i+2}+(-\d)^{-i} a^{(i+1)(q+1)}}{\delta^2+a^{q+1}}. $$

\subsection{Even characteristic}
In the even characteristic case, $c=\delta+a^{q+1}\delta^{-1}$ and
$$c^q+c=0 \Leftrightarrow \delta^{q-1}=1 \text{ or }\delta^{q+1}=a^{q+1}.$$
Hence, we can fix $\delta$ as an element in $\mathbb{F}_{q^2}$ with order $n$, which is $q-1$ or $q^2-1$.

\subsubsection{$\text{ord}(\delta)=q-1$}
If $\text{ord}(\delta)=q-1$, then we can assume that $\delta=a^{q+1}$ in this sub-subsection. Firstly let us determine the order of the automorphism $\sigma$.
Note that
$$\sigma^i=1 \Leftrightarrow v_{i-1}=0 \text{ and } v_{i}=a^i.$$
It is easy to calculate that $v_{i-1}=0 \Leftrightarrow \delta^{2i}=a^{(q+1)i}\Leftrightarrow q-1| i$. If $i=(q-1)k$ for some $k$, then 
$$v_{(q-1)k}=a^{(q-1)k}\Leftrightarrow \frac{\delta^{(q-1)k+2}+\delta^{(1-q)k}a^{((q-1)k+1)(q+1)}}{\delta^2+a^{q+1}} =a^{(q-1)k}=1
\Leftrightarrow q+1|k.$$
Hence, the order  of $\sigma$ is $\text{ord}(\sigma)=q^2-1.$

Now we consider the fixed subfield with respect to the cyclic group $\mathcal{D}$ generated by the automorphism $\sigma$.
For the infinity place $P_{\infty}$ and $1\le i \le q^2-2$,  we have
$\sigma^i(P_{\infty})=P_{\infty}  \Leftrightarrow v_{i-1}=0.$ 
For $1\le i \le q^2-2$,  we have
$\sigma^i(P_{\a,\b})=(P_{\a,\b})$ if and only if $$\frac{a^i\a}{a^{q+1}v_{i-1}\b+v_i}=\a,  \quad \frac{a^{q+1}v_{i-2}\b+v_{i-1}}{a^{q+1}v_{i-1}\b+v_i}=\b,  \quad \b^q+\b=\a^{q+1}.$$
{\bf Case 1: $\alpha \neq 0.$}\\
We have $a^{q+1}v_{i-1}\b+v_i=a^i.$ Assume that $v_{i-1}=0$, then $v_i=a^i$. It follows that $\sigma^i=1$ which is impossible. Hence, $v_{i-1}\neq 0$ and
$$\frac{a^{q+1}v_{i-2}\b+v_{i-1}}{a^{q+1}v_{i-1}\b+v_i}=\b  \Rightarrow v_{i-1}(a^{q+1}\b^2+c\b+1)=0
 \Rightarrow (a^{q+1}\b+\delta)(\b+\frac{1}{\delta})=1.$$
Thus there are two solutions $\b=\delta^{-1} \text{ and } \b=a^{-(q+1)}\delta $. It is easy to calculate that 
$$a^{q+1}v_{i-1}\b+v_i=\begin{cases}
\d^i & \text{ if } \b=\delta^{-1},\\
1 & \text{ if } \b=a^{-(q+1)}\delta.
\end{cases}$$
It is easy to check that $a^{q+1}v_{i-1}\b+v_i\neq a^i$. Therefore, 
the place $P_{\a,\b}$ with $\a\neq 0$ can't be stabilized by the automorphism $\sigma^i$ for $1\le i\le q^2-2$.\\
{\bf Case 2: $\alpha =0.$}\\
It follows that  $v_{i-1}(a^{q+1}\b^2+c\b+1)=0 \text{ and } \b^q+\b=0.$
If $v_{i-1}=0$, then $i=(q-1)k$ for $1\le k\le q$. For each $k$, the places $P_{0,\b}$ with $\b^q+\b=0$ and $P_{\infty}$ are stabilized by the automorphism $\sigma^i$. Hence, the number of the rational places stabilized by the automorphism $\sigma^i$ is $N(\sigma^i)=q+1$.\\
If $v_{i-1}\neq 0$, then $\b=\delta^{-1} \text{ or } \b=a^{-(q+1)}\delta.$
Moreover, it can be calculated directly that $$(\frac{1}{\delta})^q+\frac{1}{\delta}=\frac{\delta^{q-1}+1}{\delta^q}=0 \text{ and } (a^{-(q+1)}\delta)^q+a^{-(q+1)}\delta=a^{-(q+1)}(\delta^q+\delta)=0.$$
Hence,  the places $P_{0,\delta^{-1}}$ and $P_{0, a^{-(q+1)}\delta}$ are stabilized by the automorphism $\sigma^i$ with $v_{i-1}\neq 0$, that is, $N(\sigma^i)=2.$
By the Hurwitz genus formula,
$$q^2-q-2\geq (q^2-1)[2g(H^{\mathcal D})-2]+(q+1)q+2(q^2-2-q).$$
Hence, the genus of the fixed subfield is $g(H^{\mathcal D})=0$ and all places of degree $3$ of $H$ are unramified in $H/H^{\mathcal D}. $

\begin{theorem}
\label{a0c-}
Let $H$ be the Hermitian function field over $\mathbb{F}_{q^2}$ with even characteristic, let $m$ be a positive divisor of $q^2-1$
and let ${\mathcal G}$ be the group generated by the automorphism $\sigma^{\frac{q^2-1}{m}}$ with $d=\gcd(m,q+1)$. Then the genus of the fixed field $H^{\mathcal G}$ is $$g(H^{\mathcal G})=\frac{(q-1)(q+1-d)}{2m}.$$
\end{theorem}
\begin{proof}
It is easy to check that the number of automorphisms in the intersection of ${\mathcal G}$ and $\langle \sigma^{q-1} \rangle$ is $d=\gcd(m,q+1).$
Hence, this theorem follows immediately from the Hurwitz genus formula $$q^2-q-2=m[2g(H^{\mathcal G})-2]+(q+1)(d-1)+2(m-d).$$
\end{proof}

\subsubsection{$\delta^{q+1}=a^{q+1}$}
If $\delta^{q+1}=a^{q+1}$, then we can assume that $\delta=a$ in this sub-subsection.
Hence, $c=\delta+\delta^q$ and $$v_i=\frac{\delta^{i+1}+\d^{q(i+1)}}{\d+\d^q}.$$
Now we can determine the order of $\sigma$, since
$\sigma^i=1 \Leftrightarrow v_{i-1}=0 \text{ and } v_{i}=a^i.$
It is easy to see that $v_{i-1}=0 \Leftrightarrow \delta^{i}+\d^{qi}=0 \Leftrightarrow \d^{i(q-1)}=1 \Leftrightarrow q+1|i$. For $i=(q+1)k$, we have
$$v_i=a^i \Leftrightarrow \frac{\d^{(q+1)k+1}+\d^{q((q+1)k+1)}}{\d+\d^q}=\d^{(q+1)k}=a^{(q+1)k}.$$
Hence, $\text{ord}(\sigma)=q+1.$

Then we consider the fixed subfield with respect to the cyclic group $\mathcal{D}$ generated by the automorphism $\sigma$.
For $1\le i\le q$, we know $v_{i-1}\neq 0$ and $\sigma^i(P_{\infty})\neq P_{\infty}. $ Moreover,
$\sigma^i(P_{\a,\b})=P_{\a,\b}$ if and only if  $$\frac{a^i\a}{a^{q+1}v_{i-1}\b+v_i}=\a, \quad \frac{a^{q+1}v_{i-2}\b+v_{i-1}}{a^{q+1}v_{i-1}\b+v_i}=\b, \quad \b^q+\b=\a^{q+1}.$$
It follows from the second equation that $a^{q+1}\b^2+c\b+1=0$. Hence, $$\b=\d^{-1} \text{ or } \b=\d^{-q}.$$
If $\b=\d^{-1}$, then $a^{q+1}v_{i-1}\b+v_i=a^{q+1}v_{i-1}\d^{-1}+v_i=\d^i=a^i$ and $(\d^{-1})^q+\d^{-1}=\d^{-(q+1)}(\d^q+\d)\neq 0.$\\
If $\b=\d^{-q}$, then $a^{q+1}v_{i-1}\b+v_i=a^{q+1}v_{i-1}\d^{-q}+v_i=\d^{iq}\neq a^i$ and $(\d^{-q})^q+\d^{-q}=\d^{-(q+1)}(\d^q+\d)\neq 0.$\\
Thus the places $P_{\a,\d^{-1}}$ with $\a^{q+1}=\d^{-q}+\d^{-1}$ are stabilized by the automorphism $\sigma^i$. Hence, $N(\sigma^i)=q+1$ for $1\le i \le q$. By the Hurwitz genus formula,
$$q^2-q-2\geq (q+1)[2g(H^{\mathcal D})-2]+q(q+1).$$
Hence, the genus of the fixed subfield is $g(H^{\mathcal D})=0$ and all places of degree $3$ of $H$ are unramified in $H/H^{\mathcal D}. $

\begin{theorem}
\label{a0c+}
Let $H$ be the Hermitian function field over $\mathbb{F}_{q^2}$ with even characteristic, let $m$ be a positive divisor of $q+1$
and let ${\mathcal G}$ be the group generated by the automorphism $\sigma^{\frac{q+1}{m}}$. Then the genus of the fixed field $H^{\mathcal G}$ is $$g(H^{\mathcal G})=\frac{(q-1)(q+1-m)}{2m}.$$
\end{theorem}
\begin{proof}
This theorem follows from the Hurwitz genus formula
$$q^2-q-2=m[2g(H^{\mathcal G})-2]+(m-1)(q+1).$$
\end{proof}

\subsection{Odd characteristic}
In the odd characteristic case, $c=\delta-a^{q+1}\delta^{-1}$ and
$$c^q+c=0 \Leftrightarrow  \delta^{q+1}=a^{q+1} \text{ or } \delta^{q-1}=-1.$$

\subsubsection{$\d^{q-1}=-1$}
If $\d^{q-1}=-1$, then   we assume that  $\delta=a^{\frac{q+1}{2}}$ and $\text{ord}(\delta)=2(q-1)$  in this sub-subsection.
It can be calculated directly that $$v_{i-1}=0 \Leftrightarrow \d^{i+1}+(-1)^{i-1}\frac{a^{(q+1)i}}{\d^{i-1}}=0 \Leftrightarrow \d^{2i}=(-1)^i a^{(q+1)i} \Leftrightarrow 2|i.$$
For the even integer $i$, 
$$v_i=a^i \Leftrightarrow \d^{i+2}+(-1)^i\frac{a^{(i+1)(q+1)}}{\d^i}=a^i(\d^2+a^{q+1})\Leftrightarrow \d^i=a^i \Leftrightarrow 2(q+1)|i.$$
Hence, $\text{ord}(\sigma)=2(q+1).$

Now we consider the cyclic group $\mathcal{D}$ generated by the automorphism $\sigma$.
For the infinity place $P_{\infty}$ and $1\le i \le 2q+1$,  we have
$\sigma^i(P_{\infty})=P_{\infty}  \Leftrightarrow v_{i-1}=0.$ 
For $1\le i\le 2q+1$,  we have
$\sigma^i(P_{\a,\b})=P_{\a,\b}$ if and only if
$$\frac{a^i\a}{a^{q+1}v_{i-1}\b+v_i}=\a, \quad \frac{a^{q+1}v_{i-2}\b+v_{i-1}}{a^{q+1}v_{i-1}\b+v_i}=\b, \quad \b^q+\b=\a^{q+1}.$$
{\bf Case 1: $\a\neq 0$.}\\
In this case, $a^{q+1}v_{i-1}\b+v_i=a^i$. Assume that $v_{i-1}=0$, then $v_i=a^i$. It follows that $\s^i=1$ which is impossible for $1\le i \le 2q+1$.
Hence, $v_{i-1}\neq 0$ and
$$v_{i-1}(a^{q+1}\b^2+c\b-1)=0 
 \Rightarrow (a^{q+1}\b+\d)(\b-\d^{-1})=0
 \Rightarrow  \b=-a^{-(q+1)}\d \text{ or } \b=\d^{-1}$$
 from the second equation.
It is easy to check that
$$a^{q+1}v_{i-1}\b+v_i=\begin{cases}
\d^i & \text{ if } \b=\delta^{-1},\\
({-\d})^i  & \text{ if } \b=-a^{-(q+1)}\delta.
\end{cases}$$
It can be verified directly that $a^{q+1}v_{i-1}\b+v_i\neq a^i$. 
Hence, the places $P_{\a,\b}$ with $\a\neq 0$ can't be stabilized by the automorphism $\sigma^i$ for $1\le i \le 2q+1$. \\
{\bf Case 2: $\a= 0$.}\\
From the second equation, we have $$v_{i-1}(a^{q+1}\b^2+c\b-1)=0.$$
If $v_{i-1}=0$, then the places $P_{0,\b}$ with $\b^q+\b=0$ and $P_{\infty}$ are stabilized by the automorphism $\sigma^i$. Hence, $N(\sigma^i)=q+1$ for each even integer $i$.\\
If $v_{i-1}\neq 0$, that is, $i$ is odd, then $a^{q+1}\b^2+c\b-1=0.$ Hence, $ \b=\d^{-1} \text{ or }\b=-a^{-(q+1)}\d.$
It follows that $\b^q+\b=0$, since $\d^{q-1}=-1$. 
Hence, the places $P_{0,\d^{-1}}$ and  $P_{0,-a^{-(q+1)}\d}$ are stabilized by the automorphisms $\sigma^i$, that is, $N(\sigma^i)=2$  for each odd integer $i$. By the Hurwitz genus formula,
$$q^2-q-2\geq 2(q+1)[2g(H^{\mathcal D})-2]+(q+1)q+2(q+1).$$
Hence, the genus of fixed subfield is $g(H^{\mathcal D})=0$ and all places of degree $3$ of $H$ are unramified in $H/H^{\mathcal D}. $

\begin{theorem}
\label{oa0c-}
Let $H$ be the Hermitian function field over $\mathbb{F}_{q^2}$ with odd characteristic, let $m$ be a positive divisor of $2(q+1)$
and let ${\mathcal G}$ be the group generated by the automorphism $\sigma^{\frac{2(q+1)}{m}}$.
If $m|q+1$, then the genus of the fixed field $H^{\mathcal G}$ is $$g(H^{\mathcal G})=\frac{(q-1)(q+1-m)}{2m}.$$
Otherwise, the genus of the fixed field $H^{\mathcal G}$ is $$g(H^{\mathcal G})=\frac{(q-1)(q+1-\frac{m}{2})}{2m}.$$
\end{theorem}
\begin{proof}
If $m|q+1$, then $\frac{2(q+1)}{m}$ is even. By the Hurwitz genus formula,
$$q^2-q-2=m[2g(H^{\mathcal G})-2]+(q+1)(m-1).$$
If $m\nmid q+1$, then $\frac{2(q+1)}{m}$ is odd. By the Hurwitz genus formula,
$$q^2-q-2=m[2g(H^{\mathcal G})-2]+(q+1)(\frac{m}{2}-1)+2\cdot \frac{m}{2}.$$
This theorem follows immediately.
\end{proof}

\subsubsection{$\d^{q+1}=a^{q+1}$}
If $\d^{q+1}=a^{q+1}$, then we assume that $\d=a$ in this sub-subsection.  Firstly let us determine the order of the automorphism $\sigma$. Note that
$\sigma^i=1 \Leftrightarrow v_{i-1}=0\text{ and } v_{i}=a^i.$
It can be calculated that 
$$v_{i-1}=0  \Leftrightarrow \delta^{i+1}+(-1)^{i-1} \frac{a^{i(q+1)}}{\d^{i-1}}=0 
 \Leftrightarrow \d^{2i}=(-1)^i a^{i(q+1)} \Leftrightarrow a^{(q-1)i}=(-1)^i.$$
If $q\equiv 1 (\text{mod} 4)$, then $v_{i-1}=0 \Leftrightarrow \frac{q+1}{2}|i.$
It is easy to verify that $v_{\frac{q+1}{2}}=a^{\frac{q+1}{2}}.$
Hence, the order of $\sigma$ is $$\text{ord}(\sigma)=(q+1)/2.$$
If $q\equiv 3 (\text{mod} 4)$, then $v_{i-1}=0 \Leftrightarrow q+1|i.$ It is easy to verify that $v_{q+1}=a^{q+1}.$
Hence, the order of $\sigma$ is $$\text{ord}(\sigma)=q+1.$$

Let ${\mathcal D}$ be the cyclic group generated by the automorphism $\sigma$.
Here we assume that $3\nmid (q+1)$, then all places of degree $3$ of $H$ are unramified in $H/H^{\mathcal{D}}$.
For $1\le i\le \text{ord}(\sigma)-1$, then $v_{i-1}\neq 0$ and $\sigma^i(P_{\infty})\neq P_{\infty}.$
Hence, $\sigma^i(P_{\a,\b})=P_{\a,\b}$ if and only if
$$\frac{a^i\a}{a^{q+1}v_{i-1}\b+v_i}=\a, \quad \frac{a^{q+1}v_{i-2}\b+v_{i-1}}{a^{q+1}v_{i-1}\b+v_i}=\b, \quad \b^q+\b=\a^{q+1}.$$
It is easy to see that $a^{q+1}\b^2+c\b-1=0$, since $v_{i-1}\neq 0$. It follows that $$\b=\d^{-1} \text{ or } \b=-a^{-(q+1)}\d.$$
If $\b=\d^{-1}$, then $a^{q+1}v_{i-1}\b+v_i=\d^i=a^i$ and $\b^q+\b=\d^{-q}+\d^{-1}=\d^{-q-1}(\d^q+\d)\neq 0.$
If $\b=-a^{-(q+1)}\d$, then $a^{q+1}v_{i-1}\b+v_i=-\d v_{i-1}+v_i=(-1)^ia^{iq}\neq a^i$ and $ \b^q+\b=-a^{-(q+1)}(\d^q+\d)\neq 0.$\\
Thus the places $P_{\a,\d^{-1}}$ with $\a^{q+1}=\d^{-q}+\d^{-1}$ are stabilized by the automorphism $\sigma^i$. Hence, $N(\sigma^i)=q+1$ for $1\le i \le \text{ord}(\s)-1.$
By the Hurwitz genus formula,
$$q^2-q-2=\text{ord}(\s)\cdot [2g(H^{\mathcal D})-2]+(\text{ord}(\s)-1)(q+1).$$
Hence, the genus of the fixed subfield is
$$g(H^{\mathcal D})=\begin{cases}
    (q-1)/2 & \text{ if } q\equiv 1 (\text{mod } 4) \\
     0   & \text{ if } q\equiv 3 (\text{mod } 4).
\end{cases}$$

\begin{theorem}
\label{oa0c+}
Let $H$ be the Hermitian function field over $\mathbb{F}_{q^2}$ with odd characteristic. Assume that $3\nmid (q+1)$.  Let $m$ be a positive divisor of $\text{ord}(\s)$
and let ${\mathcal G}$ be a subgroup of ${\mathcal D}$ with order $m$. Then the genus of the fixed field of $H^{\mathcal G}$ is
$$g(H^{\mathcal G})=\frac{(q-1)(q+1-m)}{2m}.$$
\end{theorem}
\begin{proof}
This theorem follows from the Hurwitz genus formula $$q^2-q-2=m [2g(H^{\mathcal G})-2]+(m-1)(q+1).$$
\end{proof}

\begin{remark}
The places of degree $3$ of $H$ may be ramified only if $q\equiv 1 (\text{mod } 4)$ and $3| (q+1)$ hold true at the same time.
Hence, we can only assume that $q\not \equiv 5 (\text{mod } 12)$  in Theorem \ref{oa0c+} by the Chinese Remainder Theorem.
\end{remark}


\begin{thebibliography}{99}



\bibitem{AQ04} M. Abdon and L. Quoos, {\it On the genera of subfields of the Hermitian function field}, Finite Fields Appl. {\bf 10}(2004), 271--284.

\bibitem{BMXY13} A. Bassa, L.M. Ma, C.P. Xing and S.L. Yeo, {\it Towards a characterization of subfields of the Deligne--Lusztig function fields}, Journal of Combinatorial Theory, Series A {\bf 120}(2013), 1351--1371.


\bibitem{CK99} A. Cossidente and G. Korchm\'{a}ros, {\it On curves covered by the Hermitian curves}, J. Algebra {\bf 216}(1999), 56--76.


\bibitem{CKT99} A. Cossidente, G. Korchm\'{a}ros and F. Torres, {\it Curves of large genus covered by the Hermitian curve}, Comm. Algebra {\bf 28}(2000), 4707--4728.

\bibitem{FGT97} R. Fuhrmann, A. Garcia and F. Torres, {\it On maximal curves}, J. Number Theory {\bf 67}(1997), 29--51.

\bibitem{FT96} R. Fuhrmann and F. Torres, {\it The genus of curves over finite fields with many rational points}, Manuscripta Math. {\bf 89}(1996), 103--106.


\bibitem{GSX00} A. Garcia, H. Stichtenoth and C.P. Xing, {\it On subfields of the Hermitian function fields}, Compositio Mathematica {\bf 120}(2000), 137--170.

\bibitem{GK09} M. Giuliette and G. Korchm\'{a}ros, {\it A new family of maximal curves over a finite field}, Math. Ann. {\bf 343}(2009),
229--245.

\bibitem{HKT08} J.W.P. Hirschfeld, G. Korchm$\acute{a}$ros and F. Torres, {\it Algebraic Curves over a Finite Field}, Princeton Series in Applied Mathematics, Princeton University Press, 2008.


\bibitem{La87} G. Lachaud, {\it Sommes d'Eisenstein et nombre de points de certaines courbes alg$\acute{e}$briques sur les corps finis}, C. R. Acad. Sci. Paris Ser. I {\bf 305}(1987), 729--732.



\bibitem{MXY16} L.M. Ma, C.P. Xing and S.L. Yeo, {\it On automorphism groups of cyclotomic function fields over finite fields}, J. Number Theory {\bf 169}(2016), 406--419.

\bibitem{NX01} H. Niederreiter and C.P. Xing,  {\it Rational points on curves over finite fields: Theory and Applications},
 LMS {\bf 285}, Cambridge, 2001.

\bibitem{RS94} H.-G. R$\ddot{u}$ck and H. Stichtenoth, {\it A characterization of Hermitian function fields over finite fields}, J. Reine Angew. Math. {\bf 457}(1994), 185--188.

\bibitem{St73}  H. Stichtenoth, {\it $\ddot{U}$ber die Automorphismengruppe eines algebraischen
Funktionenk$\ddot{o}$rpers von Primzahlcharakteristik, Teil I and Teil II}, Arch. Math. {\bf 24}(1973), 524--544 and 615--631.

\bibitem{St09} H. Stichtenoch, {\it Algebraic Function Fields and Codes}, Grad. Texts in Math. {\bf 254}, Springer--Verlag, 2009.

\bibitem{XS95} C.P. Xing and H. Stichtenoth, {\it The genus of maximal function fields over finite fields}, Manuscript Math. {\bf 86}(1995), 217--224.




\end{thebibliography}
\end{document}